\documentclass[12pt]{amsart}
\usepackage{amssymb} 

\newtheorem{thm}{Theorem}[section]
\newtheorem{cor}[thm]{Corollary}
\newtheorem{lem}[thm]{Lemma}

\theoremstyle{definition}

\numberwithin{equation}{section}

\begin{document}
\title[Characterization of $p$-groups by sum of element orders ]
{Characterization of $p$-groups by sum of element orders}%
\author[S. M. Jafarian Amiri and Mohsen Amiri  ]{S. M. Jafarian Amiri and Mohsen Amiri }%
\address{}%
\email{}%
\email{}
\subjclass[2000]{20D15}
\keywords{ $p$-groups, element orders.}%
\thanks{}
\thanks{}
\maketitle
\begin{abstract}
Let $G$ be a finite group. Then we denote $\psi(G) = \sum_{x\in G}o(x)$ where $o(x)$ is the order of the element $x$ in $G$. In this paper we characterize some finite $p$-groups ($p$ a prime) by $\psi$ and their orders.


\end{abstract}
\maketitle
\section{\bf Introduction and Main results}
In what follows all groups are finite and $p$ is a prime.

Given a  finite group $G$, let  $\psi(H) = \sum_{x\in H} o(x)$ for $H\subseteq G$,
where as usual, $o(x)$ is the order of the element $x$. In this
note, we  ask what information about  some classes of $p-$groups $G$
can be recovered if we know both $\psi(G)$ and $|G|$.  The
starting point for the function $\psi$ is given by the paper \cite{A.J.I} which
investigates the maximum of $\psi$ among all groups of the same
order. In \cite{AJ} the authors determined the structure of the groups
which have the minimum sum of the element orders on all groups of the same order.


Let $CP_2$ be the class of finite groups $G$ such that $o(xy)\leq max\{o(x), o(y)\}$ for  all $ x,y \in G$. We denote $\Omega_{i}(G)=<\{x\in G| \ x^{p^{i}}=1\}>$ for all $i\in \mathbb{N}$. Now we state the first main result as follows.

 \begin{thm}
Suppose that  $P$ and $Q$ are contained  in  $CP_{2}$ of the same order $p^n$. Then the following statements are equivalent:
   \begin{enumerate}
     \item  $\psi(P)=\psi(Q)$.
     \item $|\Omega_{i}(P)|=|\Omega_{i}(Q)|$ for all $i\in \mathbb{N}$.
     \item $\psi(\Omega_{i}(P))=\psi(\Omega_{i}(Q))$ for all $i\in \mathbb{N}$.
   \end{enumerate}
\end{thm}

Note that the class $CP_2$ of $p$-groups is more large than the class of abelian $p$-groups, regular $p-$groups ( See Theorem 3.14 of \cite{Su},II, page 47) and $p-$groups whose subgroup lattices are modular ( see Lemma 2.3.5 of \cite{Sch}). Moreover by the main theorem in \cite{W}, we infer that powerful $p-$groups for $p$ odd also belong to $CP_{2}$.

The following is the second main result.
\begin{thm} Let  $P$ and $Q$ be two  finite  $p-$groups of the same order and  $\Omega_{m-1}(P)\neq P$, where $exp(P)=p^{m}$. If $exp(P)>exp(Q)$,
    then $\psi(P)>\psi(Q)$.
\end{thm}
In general, it is not true that if $P$ and $Q$ are $p$-groups of the same order such that $exp(P)>exp(Q)$
, then $\psi(P)>\psi(Q)$. For example consider $Q=(C_{4})^{4}$ and $P=D_{16}\times (C_{2})^{4}$. The authors
would like to thank Prof. E. Khukhru for giving this example.

But if $exp(P)=exp(Q)$, then we have the following.
\begin{thm} Let  $P$ and $Q$ belong to $CP_{2}$ of the same order and the same exponent $p^{m}$. Also suppose that  $|\Omega_{m-i}(P)|=|\Omega_{m-i}(Q)|$ for $i=1,2,..,t$. If $|\Omega_{m-t-1}(P)|<|\Omega_{m-t-1}(Q)|$, then $\psi(P)>\psi(Q)$.
\end{thm}
As an application of Theorems 1.1 and 1.2 we have the following.
\begin{thm} Let  $P$ and $Q$ belong to $CP_{2}$ of the same order $p^{n}$. Then $\psi(P)=\psi(Q)$ if and only if  there is a bijection $f:P\rightarrow Q$ such that $o(f(x))=o(x)$ for all $x\in P$.
\end{thm}

\section{ Proof of the main results    }
\begin{lem} \label{h} Let
  $P$ be a finite  $p-$group, $exp(P)=p^{m}$ and $M=\Omega_{m-1}(P)\neq P$. Then $\psi(P)=\psi(M)+|M|p^{m}(\frac{|P|}{|M|}-1)$.
\end{lem}

\begin{proof} Suppose that $X$ is a left transversal to $M$ in $P$ containing identity element.
For all   $m\in M$ and $1\neq x\in X$, we have  $o(xm)= p^{m}$. Therefore
$\psi(xM)=|M|p^n$ for all $1\neq x \in X$.
This completes the proof.
\end{proof}

\begin{thm}\label{i} Let  $P$ and $Q$ be two  finite  $p-$groups of  order $p^{n}$ and $exp(P)=p^{m}$. If $\Omega_{m-1}(P)\neq P$ and  $exp(P)>exp(Q)$,  then $\psi(P)>\psi(Q)$.
\end{thm}

\begin{proof}
Let $M=\Omega_{m-1}(P)$. Then
\begin{eqnarray*}
\psi(P)&=&\psi(M)+|M|p^{m}(\frac{|P|}{|M|}-1)\\
&>&|M|p^{m}(\frac{|P|}{|M|}-1)\\
&=&p^{m}(|P|-|M|)\\
&=&p^{m}(p^{n}-|M|)\\
&\geq& p^{m}(p^{n}-p^{n-1})\\
&\geq&p^{n}p^{m-1}.
\end{eqnarray*}

Since $exp(Q)\leq p^{m-1}$, we have $\psi(Q)<p^{n}p^{m-1}<\psi(P)$.
\end{proof}
We observe that if a finite group $G$ belongs to $CP_{2}$, then for every $x,y\in G$ satisfying $o(x)\neq  o(y)$ we have $o(xy)=max\{o(x),o(y)\}$.

We shall need the following  theorem about the groups belonging to $CP_{2}$.

\begin{thm}( See Theorem D in \cite{Deb}) \label{a} A finite group $G$ is contained in $CP_{2}$ if and only if one of the following statements holds:

\begin{enumerate}
  \item $G$ is a $p-$group and $\Omega_{n}(G)=\{x\in G\ | \ x^{p^{n}}=1\}$.
  \item $G$ is a Frobenius group of order $p^{\alpha}q^{\beta}$, $p<q$, with kernel $F(G)$ of order $p^{\alpha}$ and cyclic complement.
\end{enumerate}

\end{thm}
In the sequel assume that $P$ and $Q$ are $p$-groups belonging to $CP_2$.

\begin{lem}\label{b}
If  $|\Omega_{1}(P)|=p^r$, then $\psi(P)=1-p+p^{r+1}\psi(\frac{P}{\Omega_{1}(P)})$.
\end{lem}
\begin{proof}

Suppose that $\Omega_{1}(P)=N$. Then we have $<x>\cap N\neq 1$ for all $1\neq x \in P$, Since   $<x^{\frac{o(x)}{p}}>$ is a subgroup of $<x>\cap N$ . Let  $X$ be a left  transversal to $N$ in  $P$ such that $1\in X$. Suppose that  $1\neq x\in X$. Then  $o(x)\geq p^{2}$ since $N$ does not contain $x$. If  $y\in N$, then  by Theorem \ref{a} part one  $exp(N)=p$  and so we have
 $o(xy)=o(x)$. This implies that  $$\psi(P)=\sum_{x\in X}\psi(xN)=\psi(N)+|N|\sum_{1\neq x\in X}o(x)$$ If $1\neq x\in X$, then $<x>\cap N\neq 1$ which follows that $o(x)=po(xN)$. Hence

\begin{eqnarray*}
\psi(P)&=&\psi(N)+|N|\sum_{1\neq x\in X}o(x)\\
&=&\psi(N)+|N|p\sum_{1\neq x\in X}o(x N)\\
&=&\psi(N)+|N|p(\psi(\frac{P}{N})-1).
\end{eqnarray*}

 Since  $|N|=p^{r}$, we have  $\psi(N)=p^{r+1}-p+1$, which  completes the proof.
\end{proof}

\begin{lem}\label{c} If  $\psi(P)=\psi(Q)$, then $|\Omega_{1}(P)|=|\Omega_{1}(Q)|$.
\end{lem}

\begin{proof}

Suppose that  $\Omega_{1}(P)=N$ and $\Omega_{1}(Q)=M$. If $|N|=p^r$ and $|M|=p^t$, then it follows from previous lemma that $p^{r+1}\psi(\frac{P}{N})=p^{t+1}\psi(\frac{Q}{M})$. If $r+1<t+1$, then $p^{r+1}\psi(\frac{P}{N})\equiv 0\ (\ mod \
p^{t+1})$. Since $\frac{P}{N}$ is a $p-$group, we have
$\psi(\frac{P}{N})=1+kp$ and so $\psi(\frac{P}{N})=1+kp\equiv 1\equiv 0\ (\ mod \ p)$, a contradiction. Thus $r=t$.
 \end{proof}

\begin{lem}\label{d}
We have $\Omega_{i}(\frac{P}{\Omega_{1}(P)})=\frac{\Omega_{i+1}(P)}{\Omega_{1}(P)}$ for all $i\in\mathbb{N}$.
\end{lem}
\begin{proof} Since  $\frac{\Omega_{i+1}(P)}{\Omega_{1}(P)}\leq \Omega_{i}(\frac{P}{\Omega_{1}(P)})$, it is enough to show that
  $ \Omega_{i}(\frac{P}{\Omega_{1}(P)})\leq \frac{\Omega_{i+1}(P)}{\Omega_{1}(P)}$. Suppose that $t\Omega_{1}(P) \in
  \Omega_{i}(\frac{P}{\Omega_{1}(P)})$. Then $t^{p^{i}}\in\Omega_{1}(P)$ and since $exp(\Omega_{1}(P))=p$, we have
   $t^{p^{i+1}}=1$. Therefore  $t\in \Omega_{i+1}(P)$ and so
  $t\Omega_{1}(P)\in\frac{\Omega_{i+1}(P)}{\Omega_{1}(P)}$.

 \end{proof}
\begin{cor}\label{e}  $\frac{P}{\Omega_{1}(P)}$ belongs to $CP_2$.
\end{cor}
\begin{proof}
   It follows from previous lemma that $\Omega_{i}(\frac{P}{\Omega_{1}(P)})=\frac{\Omega_{i+1}(P)}{\Omega_{1}(P)}$, for all $i\in \mathbb{N}$. Since $P$  belongs to  $CP_{2}$, we have $\Omega_{i+1}(P)=\{x\in G\ | \ x^{p^{i+1}}=1\}$  for all $i\in \mathbb{N}$. Since $x^{p^{i}}\in \Omega_{1}(P)$, we see $\frac{\Omega_{i+1}(P)}{\Omega_{1}(P)}=\{x\Omega_{1}(P)\in \frac{P}{\Omega_{1}(P)}\ | \ x^{p^{i}}\Omega_{1}(P)=\Omega_{1}(P)\}$  for all $i\in \mathbb{N}$ by Theorem \ref{a}  and so   $\frac{P}{\Omega_{1}(P)}$ is contained in $CP_{2}$.
  \end{proof}

\begin{thm}\label{j} Let  $P$ and $Q$ have the same order $p^{n}$ and the same exponent $p^{m}$ and suppose that  $|\Omega_{m-i}(P)|=|\Omega_{m-i}(Q)|$ for $i=0,1,2,..,t$. If $|\Omega_{m-t-1}(P)|<|\Omega_{m-t-1}(Q)|$, then $\psi(P)>\psi(Q)$.
\end{thm}
\begin{proof}
If $P\in CP_{2}$ and $exp(P)=p^{m}$, then for all $i<m$, $\Omega_{i}(P)\neq P$ and
$$1<\Omega_{1}(P)<\Omega_{2}(P)<...<\Omega_{m}(P)=P.$$

Note that for all $i\leq j$, we have $\Omega_{i}(\Omega_{j}(P))=\Omega_{i}(P)$.
Using Lemma \ref{h} we can get
$$\psi(P)=\psi(\Omega_{m-t}(P))+\sum_{i=1}^{t}|\Omega_{m-i}(P)|p^{m-i+1}(\frac{|\Omega_{m-i+1}(P)|}{|\Omega_{m-i}(P)|}-1)$$ and
$$\psi(Q)=\psi(\Omega_{m-t}(Q))+\sum_{i=1}^{t}|\Omega_{m-i}(Q)|p^{m-i+1}(\frac{|\Omega_{m-i+1}(Q)|}{|\Omega_{m-i}(Q)|}-1).$$
Since
$$\sum_{i=1}^{t}|\Omega_{m-i}(P)|p^{m-i+1}(\frac{|\Omega_{m-i+1}(P)|}{|\Omega_{m-i}(P)|}-1)$$

$$=\sum_{i=1}^{t}|\Omega_{m-i}(Q)|p^{m-i+1}(\frac{|\Omega_{m-i+1}(P)|}{|\Omega_{m-i}(Q)|}-1),$$
 it is enough to prove that $\psi(\Omega_{m-t}(P))>\psi(\Omega_{m-t}(Q))$. Suppose that $|\Omega_{m-t-1}(Q)|=p^{a}|\Omega_{m-t-1}(P)|$, where $a\geq 1$. By Lemma \ref{h}, we have
 \begin{eqnarray*}
\psi(\Omega_{m-t}(P))-\psi(\Omega_{m-t}(Q))\\
&=&\psi(\Omega_{m-t-1}(P))-\psi(\Omega_{m-t-1}(Q))\\
&+&p^{m-t}(|\Omega_{m-t-1}(Q)|-|\Omega_{m-t-1}(P)|)\\
&> &p^{m-t}(|\Omega_{m-t-1}(Q)|-|\Omega_{m-t-1}(P)|)-\psi(\Omega_{m-t-1}(Q))\\
&=&p^{m-t-a}|\Omega_{m-t-1}(Q)|(p^{a}-1)-\psi(\Omega_{m-t-1}(Q))\\
&\geq&p^{m-t-1}(p-1)|\Omega_{m-t-1}(Q)|-\psi(\Omega_{m-t-1}(Q))\\
&\geq& \psi(\Omega_{m-t-1}(Q))-\psi(\Omega_{m-t-1}(Q))=0.
 \end{eqnarray*}
This completes the proof.
\end{proof}

Using Lemmas \ref{b}  and \ref{c}  we can propose another proof for Corollary 6 in \cite{T}.

\begin{cor} \label{g}Let $P$ and $Q$  be  abelian $p-$groups of the same
 order.  Then $\psi(P)=\psi(Q)$ if and only if $P\cong Q$.
\end{cor}
\begin{proof}
 It is sufficient to show that if $\psi(P)=\psi(Q)$, then $P \cong Q$. We prove this by induction on $|P|$. Base step of induction is trivial. Let   $|\Omega_{1}(P)|=p^{t}$ and
$|\Omega_{1}(Q)|=p^{r}$. It follows from Lemma \ref{b}  that

$\psi(P)=1-p+p^{r+1}\psi(\frac{P}{\Omega_{1}(P)})$ and
$\psi(Q)=1-p+p^{t+1}\psi(\frac{Q}{\Omega_{1}(Q)}) $. We have $r=t$ by  Lemma \ref{c} . Therefore $\psi(\frac{P}{\Omega_{1}(P)})=\psi(\frac{Q}{\Omega_{1}(Q)})$.  So we have $\frac{P}{\Omega_{1}(P)}\cong
\frac{Q}{\Omega_{1}(Q)}$ by induction hypothesis which implies that  $P\cong Q$.

 \end{proof}

The above  result is not true for regular $p-$groups or $p-$ groups of nilpotent class 2. For example there exists regular $3-$group $P$ such that  $|P|=27$ and $exp(P)=3$, but $P$ is not abelian, so $\psi(P)=79=\psi((C_{3})^{3})$ but $P$ is not isomorphic to $(C_{3})^{3}$.

Now we are ready to prove Theorem 1.1.

  \begin{thm}\label{f}
Suppose that  $P$ and $Q$ have the same order. Then the following statements are equivalent:
   \begin{enumerate}
     \item  $\psi(P)=\psi(Q)$.
     \item $|\Omega_{i}(P)|=|\Omega_{i}(Q)|$ for all $i\in \mathbb{N}$.
     \item $\psi(\Omega_{i}(P))=\psi(\Omega_{i}(Q))$ for all  $i\in \mathbb{N}$.
   \end{enumerate}

\end{thm}

\begin{proof}

$(1)\Rightarrow (2)$. We prove by induction on $|P|$. Suppose that $P$ and $Q$ are contained in  $CP_{2}$ and  $\psi(P)=\psi(Q)$. It follows from Lemma  \ref{b}  that $\psi(P)=1-p+p^{r+1}\psi(\frac{P}{\Omega_{1}(P)})$ and
$\psi(Q)=1-p+p^{t+1}\psi(\frac{Q}{\Omega_{1}(Q)})$ where
$|\Omega_{1}(P)|=p^{r}$ and $|\Omega_{1}(Q)|=p^{t}$. Since $\psi(P)=\psi(Q)$, we obtain $\  r=t$  by Lemma \ref{c}  and so
$\psi(\frac{P}{\Omega_{1}(P)})=\psi(\frac{Q}{\Omega_{1}(Q)})$. By corollary \ref{e}  we have $\frac{P}{\Omega_{1}(P)}$ and $\frac{Q}{\Omega_{1}(Q)}$ are in  $CP_{2}$. Since $|\frac{P}{\Omega_{1}(P)}|=|\frac{Q}{\Omega_{1}(Q)}|$,  the induction assumption yields that $|\Omega_{i}(\frac{P}{\Omega_{1}(P)})|=|\Omega_{i}(\frac{Q}{\Omega_{1}(Q)})|$ for all $i\in \mathbb{N}$. Therefore $|\Omega_{i}(P)|=|\Omega_{i}(Q)|$ by Lemmas 2.3 and 2.4.

$(2)\Rightarrow (1)$. Let $exp(P)=p^m$. By Theorem \ref{i}, we have  $exp(Q)=p^m$. Since  $P$ and $Q$ are contained in  $CP_{2}$,  we have $exp(\Omega_{j}(P))=exp(\Omega_{j}(Q))=p^j$  for all $j\in \mathbb{N}$. But

      \begin{eqnarray*}
      \psi(P)&=&1+\sum_{j=1}^{m}(|\Omega_{j}(P)|-|\Omega_{j-1}(P))|p^j\\
      &=&1+\sum_{j=1}^{m}(|\Omega_{j}(Q)|-|\Omega_{j-1}(Q)|)p^j=\psi(Q),
      \end{eqnarray*}
where the second equality holds by the hypothesis (2).

$(2)\Rightarrow (3)$.    Since  $|\Omega_{i}(P)|=|\Omega_{i}(Q)|$ for all $i\in \mathbb{N}$,  we have $exp(P)=exp(Q)$. Let $exp(P)=p^m$. Since $P$ and $Q$ are contained in  $CP_{2}$, we have
$exp(\Omega_{j}(P))=exp(\Omega_{j}(Q))=p^j$  for all $j\in \mathbb{N}$. So

      \begin{eqnarray*}
      \psi(\Omega_{i}(P))&=&1+\sum_{j=1}^{i}(|\Omega_{j}(P)|-|\Omega_{j-1}(P))|p^j\\
      &=&1+\sum_{j=1}^{i}(|\Omega_{j}(Q)|-|\Omega_{j-1}(Q)|)p^j=\psi(\Omega_{i}(Q)).
      \end{eqnarray*}

      $(3)\Rightarrow (2)$. Since $\psi(\Omega_{i}(P))=\psi(\Omega_{i}(Q))$ for all  $i\in \mathbb{N}$, we have $exp(P)=exp(Q)=p^m$. Let  $M=\Omega_{m-1}(P)$ and $N=\Omega_{m-1}(Q)$. By  Lemma \ref{h} we have  $$\psi(P)=\psi(M)+|M|p^{m}(\frac{|P|}{|M|}-1)=\psi(N)+|N|p^{m}(\frac{|Q|}{|N|}-1)=\psi(Q).$$

Since $\psi(M)=\psi(N)$,  we obtain that $|N|=|M|$. By repeated use of this technique we shall reach the claimed.
This completes the proof.
 \end{proof}
Finally we prove the last main result.
\begin{thm} \label{m}Let $P$ and $Q$ have the same order $p^{n}$. Then $\psi(P)=\psi(Q)$ if and only if there is a bijection $f:P\rightarrow Q$ such that $o(f(x))=o(x)$ for all $x\in P$.
\end{thm}

\begin{proof}
It is clear that if  there is a bijection $f:P\rightarrow Q$ such that $o(f(x))=o(x)$ for all $x\in P$, then $\psi(P)=\psi(Q)$. Conversely suppose that $\psi(P)=\psi(Q)$. We proceed by induction on $n$. Base step is trivial. By Theorem \ref{i}
we have $exp(P)=exp(Q)=p^m$. It follows from Theorem \ref{f} that $\psi(\Omega_{m-1}(P))=\psi(\Omega_{m-1}(Q))$ and so by inductive hypothesis  there is a bijection $f:\Omega_{m-1}(P)\rightarrow \Omega_{m-1}(Q)$ such that $o(f(x))=o(x)$ for all $x\in \Omega_{m-1}(P)$. Theorem   \ref{f} follows that $|\Omega_{m}(P)|-|\Omega_{m-1}(P)|=|\Omega_{m}(Q)|-|\Omega_{m-1}(Q)|$ and hence there is a bijection $g$ from $\Omega_{m}(P)-\Omega_{m-1}(P)$ to
$\Omega_{m}(Q)-\Omega_{m-1}(Q)$. Define $h$ from $P$
to $Q$ by $$h(x)=
  \begin{cases}
    f(x) & \text{ $x\in \Omega_{m-1}(P) $}, \\
    g(x) & \text{otherwise}.
  \end{cases}$$
It is easily seen that $h$ is a bijection from $P$ to $Q$ such that $o(h(x))=o(x)$ for all $x\in P$, as wanted.
\end{proof}

Department of Mathematics, Faculty of Sciences, University of
Zanjan, Zanjan, Iran

E-mail addresses:

 sm$_{-}$jafarian@znu.ac.ir

 m.amiri77@gmail.com
\end{document}